\documentclass[a4paper,11pt]{amsart}
\usepackage[english]{babel}

\newtheorem{theorem}{Theorem}[section]
\newtheorem{lm}[theorem]{Lemma}

\theoremstyle{definition}

\theoremstyle{remark}
\newtheorem{remark}[theorem]{Remark}

\usepackage{amssymb}

\def\veps{\varepsilon}

\def\SS{\mathbb S}

\def\RR{\mathbb R}
\def\SS{{\mathbb S}}
\def\RN{{\mathbb R^N}}

\def\Om{\Omega}
\def\De{\Delta}
\def\Ga{\Gamma}

\def\al{\alpha}

\def\ga{\gamma}
\def\la{\lambda}

\def\om{\omega}
\def\pa{\partial}

\def\veps{\varepsilon}

\def\dist{\mbox{\rm dist}}

\newcommand{\diam}{\mathop{\mathrm{diam}}}

\def\ds{\displaystyle}
\def\ovr{\overline}

\def\nr{\Vert}
\renewcommand\emptyset{\varnothing}

\newcommand{\usn}{[u]_{\pa G}}

\begin{document}
    \title[]{Solutions of elliptic equations
with a level surface parallel to the boundary: stability of the radial configuration}

\author[G. Ciraolo]{Giulio Ciraolo}
\address{Dipartimento di Matematica e
Informatica, Universit\`a di Palermo, Via Archirafi 34, 90123, Italy.}
\email{g.ciraolo@math.unipa.it}
\urladdr{http://www.math.unipa.it/~g.ciraolo/}

\author[R. Magnanini]{Rolando Magnanini}
    \address{Dipartimento di Matematica ed Informatica ``U.~Dini'',
Universit\` a di Firenze, viale Morgagni 67/A, 50134 Firenze, Italy.}
    \email{magnanin@math.unifi.it}
    \urladdr{http://web.math.unifi.it/users/magnanin}

\author[S. Sakaguchi]{Shigeru Sakaguchi}
    \address{Research Center for Pure and Applied Mathematics, Graduate School of Information Sciences,
Tohoku University, Sendai 980-8579, Japan.}
    \email{sigersak@m.tohoku.ac.jp}
\urladdr{http://researchmap.jp/sigersak2012415/}
\begin{abstract}
Positive solutions of homogeneous Dirichlet boundary value problems or initial-value problems for certain elliptic or parabolic equations must be radially symmetric
and monotone in the radial direction if  {\it just one} of their level surfaces is parallel to the boundary of the domain. Here, for the elliptic case, we prove the stability counterpart of that result. In fact, we show that if the solution is almost constant on a surface at a fixed distance from the boundary, then the domain is almost radially symmetric, in the sense that is contained in and contains two concentric balls $B_{r_e}$ and  $B_{r_i}$, with the difference $r_e-r_i$  (linearly) controlled by a
suitable norm of the deviation of the solution from a constant. The proof relies on and enhances arguments developed by Aftalion, Busca and Reichel in \cite{ABR}.
\end{abstract}


\maketitle

\section[Introduction]{Introduction}
Let $\Om$ be a bounded domain in $\RN$. It has been noticed in \cite{MSaihp}-\cite{MSm2as}, \cite{Sh} and \cite{CMS} that positive solutions of homogeneous Dirichlet boundary value problems or initial-boundary value problems for certain elliptic or, respectively, parabolic equations must be radially symmetric (and the underlying domain $\Om$ be a ball) if  {\it just one} of their level surfaces is parallel to $\pa\Om$ (that is,
if the distance of its points from $\pa\Om$ remains constant).
\par
The property was first identified in \cite{MSaihp}, motivated by the study of {\it invariant isothermic surfaces} of a nonlinear non-degenerate {\it fast diffusion} equation (designed upon the heat equation), and was used
to extend to nonlinear equations the symmetry results obtained in \cite{MS1} for the heat equation.
The proof hinges on the {\it method of moving planes} developed by J. Serrin in \cite{Se}
upon A.D. Aleksandrov's reflection principle (\cite{Al}). Under slightly different assumptions and by a different proof still based on the method of moving planes, a similar result was independently obtained in \cite{Sh} (see also \cite{GGS}). Further extensions can be found in \cite{MSm2as} and \cite{CMS}.
\par
To clarify matters, we consider a simple situation (possibly the simplest). Let $G$ be
a $C^1$-smooth domain in $\RN$, denote by $B_R$ the ball centered at the origin with radius $R$,
\footnote{In the sequel, a ball with radius $R$ centered at a generic point $x$ will be denoted by
$B_R(x)$.}
and define the set
$$
\Om=G+B_R=\{y+z: y\in D, |z|<R\}
$$
--- the {\it Minkowski sum} of  $G$ and $B_R$. Consider the solution $u=u(x)$ of the {\it torsion} boundary value problem
\begin{eqnarray}
\label{torsion}
-\De u=1\ \mbox{ in }\ \Om, \quad u=0 \ \mbox{ on } \ \pa\Om.
\end{eqnarray}
As shown in the aforementioned references, if there is a positive constant $c$ such that
\begin{eqnarray}
\label{parallel}
u=c \ \mbox{ on } \ \pa G,
\end{eqnarray}
then $G$ and $\Om$ must be concentric balls.
\par
The aim of this paper is to investigate on the stability of the radial configuration. In other words,
we suppose that $u$ is {\it almost constant} on $\pa G$ by assuming that the semi-norm
\begin{equation*}
    \usn =  \sup_{{\substack{x,y \in \pa G, \\ x\neq y}}} \frac{|u(x)-u(y)|}{|x-y|}
\end{equation*}
is {\it small} and want to quantitatively show that $\Om$ is close to a ball. A result in this direction that concerns the torsion problem \eqref{torsion} is the following theorem.

\begin{theorem}
\label{th:torsionstability}
Let $G$ be a bounded domain in $\RR^N$ with boundary $\pa G$ of class $C^{2,\alpha}$, $0<\alpha \leq 1$, and set
$\Om=G+B_R$ for some $R>0$. Let $u\in C^2(\Om) \cap C^0({\ovr\Om})$ be the solution of  \eqref{torsion}.
\par
There exist constants $\veps, C>0$ such that, if
$$
\usn \leq \veps,
$$
then there are two concentric balls $B_{r_i}$ and $B_{r_e}$ such that
\begin{eqnarray}
&B_{r_i}\subset\Om\subset B_{r_e}\quad\mbox{ and }\label{stabil1}\\
&r_e-r_i\le C\ \usn .\label{stabil2}
\end{eqnarray}
The constants $\veps$ and $C$ only depend on $N$, the $C^{2,\alpha}$-regularity
of $\pa G$, the diameter of $G$ and, more importantly, the number $R$.
\end{theorem}
\par
A result similar to Theorem \ref{th:torsionstability}
was proved by Aftalion, Busca and Reichel in \cite{ABR}. There, under further suitable assumptions, it is proved the stability estimate
\begin{equation}
\label{ABR}
r_e-r_i\le C\,|\log  \left\|u_\nu-d\|_{C^1(\pa \Om)} \right|^{-1/N},\footnote{For the definition
of the norm, see Section 4.}
\end{equation}
where $u_\nu$ denotes the (exterior) normal derivative of the solution of \eqref{torsion},
$d$ is a constant and $\| \cdot \|_{C^1(\pa \Om)}$ is the usual $C^1$ norm on $\pa G$;
\eqref{ABR} is the quantitative version of Serrin's symmetry result \cite{Se} that stated that the
solutions of \eqref{torsion} whose normal derivative is constant on $\pa\Om$, that is
\begin{equation}
\label{neumann}
u_\nu=d \ \mbox{ on } \ \pa\Om,
\end{equation}
must be radially symmetric (and $\Om$ must be a ball). In \cite{ABR},
\eqref{ABR} is a particular case of a similar estimate that holds for general semilinear equations, that is when solutions of the problem
\begin{equation}
\label{semilinear}
-\De u=f(u) \ \mbox{ and } \ u\ge 0 \ \mbox{ in } \ \Om, \qquad u=0 \ \mbox{ on }  \ \pa\Om,
\end{equation}
are considered (here, $f$ is a locally Lipschitz continuous function).
\par
The logarithmic dependence in \eqref{ABR} is due to the method of proof employed, which is based on the idea of
refining the method of moving planes from a quantitative point of view. As that method is
based on the maximum (or comparison) principle, its quantitative counterpart is based on {\it Harnack's inequality} and some quantitative versions of {\it Hopf's boundary lemma} and {\it Serrin's corner lemma} (that involves the second derivatives of the solution on $\pa\Om$). The exponential dependence of the constant involved in Harnack's inequality leads to the
logarithmic dependence in \eqref{ABR}.
\par
Estimate \eqref{ABR} was largely improved in \cite{BNST2} (see also \cite{BNST}) but only for the case of the torsion problem \eqref{torsion} (a version of this result is also available for Monge-Amp\`ere equations \cite{BNST3}). There, \eqref{ABR} is enhanced in two respects: the logarithmic dependence at the right-hand side of \eqref{ABR} is replaced by a power law of H\"older type; an estimate is also given in which the $C^1$-norm is merely replaced by the $L^1$-norm.
The results in \cite{BNST} are obtained by remodeling in a quantitative manner the Weinberger's proof \cite{We} of Serrin's result, which is grounded on integral formulas such as Rellich-Pohozaev's identity. We observe that, unfortunately,
that approach does not seem to work in our setting, since the overdetermining condition \eqref{parallel} does not conveniently match the underlying integral formulas of {\it variational} type.

\par

The proof of Theorem \ref{th:torsionstability} adapts the approach used in \cite{ABR} but, differently
from that paper, it results in the {\it linear} estimate \eqref{stabil1}-\eqref{stabil2}.
The reason for this gain must be ascribed to the fact that, in our setting, we deal with the values of $u$
on $\pa G$, which is in the {\it interior} of $\Om$, and we are dispensed with the use
of estimates {\it up to the boundary of} $\pa\Om$, which are accountable for the logarithmic behavior in \eqref{ABR}.
However, the benefit obtained in Theorem \ref{th:torsionstability} has a cost: the constant $C$ in
\eqref{stabil2} blows up exponentially to infinity as $R$ tends to zero.
\par
The proof can be outlined as follows. For any fixed direction $\om$, the method of moving planes determines a hyperplane $\pi$, orthogonal to $\om$  and in {\it critical \footnote{Se Section 2 for the meaning of this word.} position},
 and a domain $X$ contained in $G$ and symmetric about $\pi$.  Since $\ovr{X} \subseteq \ovr{G} \subset \Omega$, we can use interior and boundary Harnack's inequalities to estimate in terms of $\usn$ the values of the harmonic function $w(x)=u(x')-u(x)$ (here, $x'$ is the {\it mirror image}  in $\pi$ of the point $x$) in the half of $X$ where $w$ is non-negative.
Such an estimate gives a bound on the distance of $\pa X$ from $\pa G$; it is important to observe that such a bound does not depend on the direction $\om$. By repeating this argument for $N$ orthogonal directions, an approximate center of symmetry --- where the two balls $B_{r_i}$ and $B_{r_e}$ in \eqref{stabil1} are centered --- can be determined. The radii of these balls then satisfy \eqref{stabil2}.
\par
We conclude this introduction with some remarks on the relationship between problem
\eqref{torsion}-\eqref{parallel} and Serrin's problem \eqref{torsion}-\eqref{neumann}. In both cases,
a solution exists (if and) only if the domain is a ball; a natural question arises: does the symmetry obtained for one of these problems imply that for the other one?
\par
We begin by observing that conditions \eqref{torsion}-\eqref{parallel} neither imply nor are implied by conditions \eqref{torsion}-\eqref{neumann}, directly. However, we can claim that \eqref{neumann} seems to be a stronger constraint, in the sense that, in order to obtain it, one has to require that
\eqref{parallel} holds (with different constants) for at least a sequence of parallel surfaces clustering around $\pa\Om$. On the other hand, if \eqref{torsion}-\eqref{neumann} holds, we cannot claim that its solution $u$ is constant on surfaces parallel to $\pa\Om$, but we can surely say that the oscillation
of $u$ on a parallel surface becomes smaller, the closer the surface is to $\pa\Om$; in fact, an easy Taylor-expansion argument tells us that, if $\Ga_t=\{x\in\Om: \dist(x,\pa\Om)=t\}$, then
\begin{equation}
\label{asymptotic}
\max_{\Ga_t}u-\min_{\Ga_t}u=O(t^2) \ \mbox{ as } \ t\to 0^+.
\end{equation}
Theorem \ref{th:torsionstability} suggests the possibility that Serrin's symmetry result may be
obtained, so to speak {\it by stability}, via \eqref{stabil2} (conveniently remodeled) and \eqref{asymptotic} in the limit as $t\to 0^+$.
If successful, this plan would show that the method of moving planes can be employed to prove Serrin's symmetry result with no need of the aforementioned corner's lemma and under weaker assumptions on $\pa \Omega$ and $u$ (see also \cite{Sh}). To our knowledge, this issue has been addressed only in \cite{Pr}, for domains with one possible corner or cusp, and in \cite{GL}, by arguments based on \cite{We}. At this moment, the blowing up dependence of the constant $C$ in \eqref{stabil2} on the distance of the relevant parallel surface from $\pa\Om$ prevents us to make this possibility come true.
\par
The paper is organized as follows. In Section \ref{section symmetry thm} we introduce some notation and, for the reader's convenience, we recall the proof of the relevant symmetry result. In Section \ref{section stability estimates}, we carry out the proof of Theorem \ref{th:torsionstability}.
In Section \ref{section semilinear}, we explain how our result can be extended to the classical case of semilinear equations \eqref{semilinear}.

\setcounter{equation}{0}

\section{Parallel surfaces and symmetry}  \label{section symmetry thm}

To assist the reader to understand the proof of Theorem \ref{th:torsionstability}, here we summarize the arguments developed in \cite{MSaihp}, \cite{MSm2as} and \cite{CMS} to prove the symmetry of a domain
$\Om$ admitting a solution $u$ of \eqref{torsion} and \eqref{parallel}. This will be the occasion
to set up some necessary notations.

In this section we assume that the boundary of $G$ is of class $C^1$; this assumption guarantees that the method of moving planes can be initialized (see \cite{Fr}).
\par
Let $\om\in\RN$ be a unit vector and $\la\in\RR$ a parameter. We define the following objects:
\begin{equation}
\label{definitions}
\begin{array}{lll}
&\pi_{\la}=\{ x\in\RN: x\cdot\om=\la\}\ &\mbox{a hyperplane orthogonal to $\om,$}\\
&A_{\la}=\{x\in A: x\cdot\om>\la\}\ &\mbox{the right-hand cap of a set $A$},\\
&x^{\la}=x-2(x\cdot\om-\la)\,\om\ &\mbox{the reflection of $x$ about $\pi_{\la},$}\\
&A^{\la}=\{x\in\RN:x^{\la}\in A_{\la}\}\ &\mbox{the reflected cap about $\pi_{\la}$.}
\end{array}
\end{equation}
Set $M=\sup\{x\cdot\om: x\in G\}$, the extent of $G$ in direction $\om$; if $\la<M$ is close to $M$, the reflected cap $G^\la$ is contained in $G$ (see \cite{Fr}). Set 
\begin{equation}\label{m def}
m=\inf\{\mu: G^{\la}\subset G \mbox{ for all } \la\in(\mu,M)\}.
\end{equation}
Then for $\lambda = m$ at least one of the following two cases occurs:
 \begin{enumerate}
\item[(i)]
$G^{\la}$ becomes internally tangent to $\pa G$ at some
point $P\in\pa G\setminus\pi_{\la};$
\item[(ii)]
$\pi_{\la}$ is orthogonal to $\pa G$ at some point $Q\in\pa G\cap\pi_{\la}.$
\end{enumerate}

We notice that if
$$
\Om=G+B_R = \{x+y:\ x \in G,\ y \in B_R\},
$$
then we have that
$$
G^{\la} \subset G  \ \Rightarrow \ \Om^{\la} \subset \Om,
$$
which follows from an easy adaptation of \cite[Lemmas 2.1 and 2.2]{MSm2as}.

\begin{theorem} \label{thm symmetry}
Let $G$ be a bounded domain in $\RR^N$ with boundary $\pa G$ of class $C^1$ and set
$\Om=G+B_R$ for some $R>0$. Suppose there exists a function $u\in C^0(\ovr{\Om})\cap C^2({\Om})$ satisfying \eqref{torsion} and \eqref{parallel}.
\par
Then $G$ (and hence $\Omega$) is a ball and $u$ is radially symmetric.\footnote{Note that, if $G$ is a spherical annulus in $\RR^N$, then the solution of \eqref{torsion} is radially symmetric but does not satisfy \eqref{parallel}.}
\end{theorem}

\begin{proof}
By contradiction, let us assume that $G$ is not a ball. Let $\om\in\SS^{N-1}$ and, on the right-hand cap $\Om_{m}$ at critical position $\la=m,$
consider the function defined by
\begin{equation} \label{wm}
w^{m}(x)=u(x^{m})-u(x), \ x\in\Om_{m}.
\end{equation}
It is clear that $w^{m}$ satisfies the conditions
\begin{equation}
\label{wmsys}
\De w^{m}=0 \ \mbox{ and } \ w^{m}\ge 0 \ \mbox{ on } \ \pa\Om_{m}.
\end{equation}
Moreover, since $G$ is not a ball, there is at least one direction $\om$
such that $w^m$ is not identically zero on $\pa\Om_m$; hence, by the strong maximum principle, we have that
\begin{equation}
\label{w > 0}
w^{m}>0 \ \mbox{ in } \ \Om_{m}.
\end{equation}
\par
Also, we can apply Hopf's boundary point lemma to the points in $\pa\Om^m\cap\pi_m$ and obtain that
\begin{equation*}
\frac{\pa w^{m}}{\pa \om} >0\ \textmd{ on } \ \pa\Om_m\cap\pi_m,
\end{equation*}
and hence
\begin{equation}
\label{hopf}
\frac{\pa u}{\pa\om}<0 \ \mbox{ on }  \ \pa\Om_m\cap\pi_m,
\end{equation}
since
$$
\ds\frac{\pa w^{m}}{\pa\om}=-2\frac{\pa u}{\pa\om} \ \textmd{ on } \ \Om\cap\pi_m.
$$
\par
If (i) holds, we get a contradiction to \eqref{w > 0} by observing that
$$
w^{m}(P)=u(P^{m})-u(P)=0,
$$
since \eqref{parallel} holds and both $P$ and $P^{m}$ belong to $\pa G$.
\par
If (ii) holds, we observe that $\om$ is tangent
to $\pa G$ at $Q$ and that \eqref{parallel} implies that
$$
\frac{\pa u}{\pa\om}(Q)=0,
$$
contradicting \eqref{hopf}.
\par
Thus, we conclude that $G$ must be a ball and hence $\Om =G + B_R$ is also a ball.
\end{proof}

\setcounter{equation}{0}

\section{Stability estimates for the torsion problem} \label{section stability estimates}
In this section, we present a proof of Theorem \ref{th:torsionstability} that
relies on the ideas in \cite{ABR} and the {\it Harnack-type} and {\it tangential} stability estimates
contained in Lemmas \ref{proposition Harnack} and \ref{proposition Harnack 2}, respectively.
\par
These Lemmas analyse the two critical situations (i) and (ii) of Section 2 from a
quantitative point of view. Thus, we shall consider two domains in $\RN$,
$D_1$ and $D_2$, containing the origin and such that $D_1\subset D_2$. We shall also
suppose that there is an $R>0$ such that $B_R(x)\subset D_2$ for any $x\in\ovr{D}_1$
and use the following notations:
\begin{equation} \label{D e Er}
\begin{array}{ll}
D^+_j=\{ x\in D_j: x_1>0\}, \ j=1, 2, \\
E_R=\{ x\in D_1+B_{R/2}: x_1>R/2\}.
\end{array}
\end{equation}
Without loss of generality, we can assume that $D^+_1$ is connected; this easily implies
that also $E_R$ is connected.

\begin{lm}\label{proposition Harnack}
Assume that a function $w\in C^0(D_2)\cap C^2(D^+_2)$ satisfies the conditions
\begin{equation*} \label{w eq}
\De w=0 \ \mbox{ and } \ w\ge 0 \ \mbox{ in } \ D^+_2, \quad w=0 \ \textmd{ on } \ \pa D_2^+ \cap \{x_1=0\} ,
\end{equation*}
and let $z=(z_1,\ldots,z_N)$ be any point  in $\ovr{D_1^+}$.
\par
Then, for every $x \in \ovr{E_R}$ it holds that
\begin{eqnarray}
&& w(x) \leq C\,\frac{w(z)}{z_1}, \quad  \textmd{if } z_1>0,  \label{estim z1 >0} \\
&& w(x) \leq C\,\frac{\partial w}{\pa x_1} (z),\quad  \textmd{if } z_1=0,  \label{estim z1 =0}
\end{eqnarray}
where $C$ is a constant depending on $N$, the diameters of $D_1$ and $D_2$, and $R$
(see \eqref{C(diam D_1,R,N)} below).
\end{lm}

\begin{proof}
We can always assume that $w>0$ in $D_2^+$.
We shall distinguish the three cases: (i) $z_1 \geq R/2$; (ii) $0<z_1<R/2$; (iii) $z_1=0$.

(i) Let $B_{r}(x_0)$ be any ball contained in $D_2^+$. It is well-known that Harnack's inequality (see \cite{GT}) gives:
\begin{equation}\label{Harnack Harmonic}
r^{N-2} \,\frac{r-|x-x_0|}{(r+|x-x_0|)^{N-1}}\, w(x_0) \leq w(x) \leq r^{N-2}\,\frac{r+|x-x_0|}{(r-|x-x_0|)^{N-1}}\, w(x_0),
\end{equation}
for any $x\in B_r(x_0)$. By the monotonicity of the two ratios in $|x-x_0|$, we have that
\begin{equation}\label{Harnack Harmonic 2}
2^{N-2}3^{1-N} w(x_0) \leq w(x) \leq 3\cdot 2^{N-2} w(x_0) \quad \textmd{for any } x \in \overline{B_{r/2}(x_0)}.
\end{equation}
\par
Let $y\in \ovr{E_R}$ be such that
$$
w(y)=\sup\limits_{E_R} w.
$$
Since $\dist(y,\pa D_2^+) \geq R/2$ and $(D_1+B_{R/2}) \cap \{x_1>s\}$ is connected for any $0 \leq s\leq R/2$, by \cite[Lemma A.1]{ABR}, there exists a chain of $n$ pairwise disjoint balls $\{ B_{R/4}(p_i) \}_{i=1}^n$ such that $\dist(p_i,\pa D_2^+) \geq R/2,\: i=1,\ldots,n$,
$$
 y,z \in \bigcup_{i=1}^n \overline{B_{R/4}(p_i)},
$$
and
$$
\overline{B_{R/4}(p_i)} \cap \overline{B_{R/4}(p_{i+1})}\neq \emptyset,\quad  i=1,\ldots,n-1. $$
Since $B_{R/2}(p_i) \subset D_2^+$, by applying \eqref{Harnack Harmonic 2} to each ball and combining the resulting inequalities yields that
\begin{equation} \label{Harnack case i 1}
\sup_{E_R} w \leq (3\cdot 2^{N-2})^{n+1} w(z).
\end{equation}
Since $z_1 \leq \diam D_1^+\le\diam D_1$, we easily obtain that
\begin{equation} \label{Harnack case i 2}
\sup_{{E}_R} w \leq (\diam D_1)(3\cdot 2^{N-2})^{n+1} \frac{w(z)}{z_1}.
\end{equation}
An upper bound for the optimal number of balls $n$ is clearly
$(2\,\diam D_2)^N R^{-N};$
the last inequality then implies \eqref{estim z1 >0}.

(ii) Since $0<z_1 < R/2$, the point $\hat{z}=(R/2,z_2,\ldots,z_N)$ is such that
\begin{equation*}
z \in B_{R/2}(\hat{z}) \subset D_2^+ .
\end{equation*}
We apply \eqref{Harnack Harmonic} with $x_0=\hat{z},\ x=z$ and $r=R/2$ and obtain that
\begin{equation*}
\Big(\frac{R}{2} \Big)^{N-2} \frac{z_1}{(R-z_1)^{N-1}}\,w(\hat{z}) \leq w(z),
\end{equation*}
and hence
\begin{equation*}
w(\hat{z}) \leq 2^{N-2}R\,\frac{w(z)}{z_1}.
\end{equation*}
By noticing that $\hat{z} \in \ovr{E_R}$, we can use the same Harnack-type argument used for case (i) (see formula \eqref{Harnack case i 1}) to obtain
$$
\sup_{E_R} w \leq (3\cdot 2^{N-2})^{[1+(2\,\diam D_2)^N/R^{N}]} w(\hat{z}),
$$
and hence
\begin{equation}\label{Harnack case ii}
 \sup_{{E}_R} w \leq 2^{N-2}R (3\cdot2^{N-2})^{[1+(2\diam D_2)^N/R^N]} \frac{w(z)}{z_1}.
\end{equation}
\par
(iii) By taking the limit as $z_1 \to 0$ in \eqref{Harnack case ii} and noticing that $w= 0$ for $x_1 = 0$, we obtain \eqref{estim z1 =0} as
\begin{equation}\label{Harnack case iii}
 \sup_{{E}_R} w \leq 2^{N-2}R (3\cdot2^{N-2})^{[1+(2\diam D_2)^N/R^N]} \frac{\pa w}{\pa x_1}(z).
\end{equation}
An inspection of \eqref{Harnack case i 2}, \eqref{Harnack case ii} and \eqref{Harnack case iii}
informs us that the constant $C$ in \eqref{estim z1 >0} and \eqref{estim z1 =0} is given by
\begin{equation}\label{C(diam D_1,R,N)}
C=3\max(2^{N-2} R, \diam D_1) 2^{N-2} C_N^{(\diam D_2 / R)^N} \ \mbox{ with } \ C_N=3^{2^N}2^{(N-2)2^N}.
\end{equation}
\end{proof}

Now, we extend estimates \eqref{estim z1 >0} and \eqref{estim z1 =0} to the case in which $x$ is any point in $\overline{D_1^+} \setminus \{x_1 = 0\}$.

\begin{lm}\label{proposition Harnack 2}
Let $D_1,\: D_2, \: R, w$ and $z$ be as in Lemma \ref{proposition Harnack}.
\par
Then, for any $x \in \overline{D_1^+} \setminus \{x_1 = 0 \}$, it holds that
\begin{equation*}
\begin{array}{ll}
\ds w(x) \leq M\,C\,\frac{w(z)}{z_1}, \quad  &\textmd{if } z_1>0,  \\ 
\ds w(x) \leq M\,C\,\frac{\partial w}{\pa x_1} (z),\quad &\textmd{if } z_1=0,  \label{estim z1 =0 II}
\end{array}
\end{equation*}
where $C$ is given by \eqref{C(diam D_1,R,N)} and $M$ is a constant only depending on $N$.
\end{lm}

\begin{proof}
For $x \in \overline{D_1^+} \cap \ovr{{E}_R}$ we clearly have that $w(x) \leq \sup_{{E}_R} w $.
\par
If $x \in\ovr{D_1^+} $ and $0<x_1<R/2$, we estimate $w(x)$ in terms of the value $w(\hat{x})$ with $\hat{x}=(R/2,x_2,\ldots,x_N)\in\ovr{{E}_R}$. In order to do this, we use the following {\it boundary Harnack's inequality} (see for instance \cite[Theorem 11.5]{CS}):
\begin{equation} \label{bound Harnack I}
\sup_{B_{R/2}^+ (x_0)} w \leq M\,w(\hat{x}),
\end{equation}
with $x_0=(0,x_2,\ldots,x_N)$; here, $M \geq 1$ is a constant only depending $N$. Since $M\geq 1$ and $w(\hat{x})\leq \sup_{{E}_R} w$, Lemma \ref{proposition Harnack} yields the conclusion.
\end{proof}

We notice that since $\pa G$ is of class $C^{2,\alpha}$, then it satisfies a uniform interior ball condition of (optimal) radius $\rho>0$ at each point of the boundary. The following result is our analog of \cite[Proposition 1]{ABR}.

\begin{theorem} \label{th:abr}
Let $\Om$ and $G$ be as in Theorem \ref{th:torsionstability} and consider a solution $u\in C^{2,\al}(\Om)$ of \eqref{torsion}.

Pick a unit vector $\om\in\RN$ and let $G_{m}$ be the maximal cap of $G$ in the direction $\om$ as defined by \eqref{definitions} and \eqref{m def}. Then for (a connected component of) 
$G_m$ we have that
\begin{equation}\label{smallness}
w^{m}\leq M C\,\usn \ \mbox{ in } \ G_m;
\end{equation}
here, $w^{m}$ is defined by \eqref{wm} and $C$ is a constant depending on $N$, $R$, the diameter of $G$
and the $C^2$-regularity of $\pa G$ (see \eqref{C3} below).
\end{theorem}

\begin{proof}
We adopt the same notations introduced for the proof of Theorem \ref{thm symmetry} and we can always assume that $\om = e_1$. The function $w=w^m$ satisfies \eqref{wmsys} and
$$
w=0 \ \mbox{ on } \ \pa G_m\cap\pi_m.
$$
Let us assume that case (i) of the proof of Theorem \ref{thm symmetry} occurs. Without loss of
generality, we shall denote by the same symbols $G_m$ and $\Om_m$, respectively, the connected component of $G_{m}$ which intersects $B_R(P^m)$ and the corresponding connected component of $\Om_m$. Modulo a translation in the direction of $e_1$, we can apply Lemma \ref{proposition Harnack 2} with $D_1^+=G_m$ and $D_2^+=\Om_m$ and $z=P^m$.
We obtain the estimate
\begin{equation} \label{w estim}
w\leq M C' \frac{w(P^m)}{\dist (P^m,\pi_{m})}\ \mbox{ in } \ G_m,
\end{equation}
where $C'$ is computed by means of \eqref{C(diam D_1,R,N)}, that is
\begin{equation*}
C'=\max(2^{N-2} R, \diam G)\,C_N^{(2 + \diam G / R)^N};
\end{equation*}
here, we have used that $\diam \Om= 2R+\diam G$.
\par
In order to estimate $w(P^m)/\dist (P^m,\pi_{m})$, we notice that $|P-P^{m}| = 2\,\dist(P^m,\pi_{m})$ and we distinguish two cases.
\par
If $|P-P^m|\geq \rho$, since $P$ and $P^{m}$ lie on $\pa G$, then
\begin{equation*}
w(P^m) = u(P)-u(P^m) \leq \diam(G) \, \usn ,
\end{equation*}
and hence we easily obtain that
\begin{equation*}\label{estim I}
\frac{w(P^m)}{\dist (P^m,\pi_{m})} \leq \frac{2 \diam(G)}{\rho}\, \usn.
\end{equation*}
If instead $|P-P^{m}| < \rho$, every point of the segment joining $P$ to $P^m$ is at distance less than $\rho$ from a connected component of $\pa G$. The curve $\gamma$ obtained by projecting that segment on $\pa G$ has length bounded by $C''\,|P-P^m|$, where $C''$ is a constant depending on $\rho$ and on the regularity of $\pa G$. An application of the mean value theorem to the function $u$ restricted to $\gamma$ gives that $u(P)-u(P^m)$ can be estimated by the length of $\ga$ times the maximum of the tangential gradient of $u$ on $\pa G$. Thus,
\begin{equation}\label{estim II}
w(P^m) \leq 2\,C''\,\dist(P^m,\pi_{m})\,\usn.
\end{equation}
From \eqref{C(diam D_1,R,N)} and \eqref{w estim}--\eqref{estim II} we then get \eqref{smallness}.

Now, let us assume that case (ii) of the proof of Theorem \ref{thm symmetry} occurs.
Again, without loss of generality, we denote by the same symbol $G_m$ the connected component of $G_m$ which intersects $B_R(Q)$, and $\Om_m$ accordingly. Then, we apply Lemma \ref{proposition Harnack 2} with $D_1^+=G_m$, $D_2^+=\Om_m$ and $z=Q$, and obtain that
\begin{equation*}
w(x) \leq M\,C\,\frac{\pa w(Q)}{\pa \omega},
\end{equation*}
for all $x\in G_m$. Since $\om$ belongs to the tangent hyperplane to $\pa G$ at $Q$, we again obtain \eqref{smallness}.
\par
An inspection of the calculations in this proof informs us that the constant $C$ in \eqref{smallness}
can be chosen as
\begin{equation} \label{C3}
C=\max\left(1,\frac{2\diam G}{\rho},2\rho\,C''\right) \max\left(R,\frac{\diam G}{2^{N-2}}\right)\,C_N^{[2+\frac{\diam G}{R}]^N}.
\end{equation}
\end{proof}

From now on, the proof of our Theorem \ref{th:torsionstability} follows the
arguments used in \cite[Sections 3 and 4]{ABR}, with a major change (we use our Theorem \ref{th:abr} instead of \cite[Proposition 1]{ABR}) and some minor changes that we shall sketch below.
\par
The key idea is to use the smallness of $w^m$ in (the relevant connected component of) $G_m$ to show that $G$ is almost equal to the symmetric open set $X$ which is defined as the interior of
$$
G_m\cup G^{m} \cup (\pa G_m \cap \pi_m).
$$
In order to do that, we need {\it a priori} bounds on $u$ from below in
terms of the distance function from $\pa G$. In \cite{ABR}, such bounds are derived in Proposition 4, where it is proven that, if $u=0$ on $\pa \Om$, then
there exist positive constants $K_1$ and $K_2$ such that
\begin{equation}
\label{boundsdist}
K_1\, \dist(x,\pa\Om)\le u(x)\le K_2\, \dist(x,\pa\Om)\ \mbox{ for all }\ x\in\Om;
\end{equation}
$1/K_1$ and $K_2$ are bounded by a constant depending on the diameter of $\Om$ and the $C^{2,\alpha}$-regularity of $\pa\Om$.
\par
However, if we examine our situation, we notice two facts: on one hand, differently from \cite{ABR}, our symmetrized set $X$ is the reflection of a connected component of $G$ --- compactly contained in $\Om$
at distance $R$ from $\pa\Om$ --- and we need not extend the relevant estimates up to the boundary of $\Om$, thus obtaining a better rate of stability. On the other hand, since we are dealing with the distance from $\pa G$ instead of that from $\pa \Om$ and $u$ {\it is not constant on} $\pa G$, we need a stability version of the first inequality in \eqref{boundsdist}. This is done in the following lemma.

\begin{lm} \label{lemma bound on dist}
There exists a positive constant $K$ depending on $\diam G$ and the $C^{2,\alpha}$ regularity of $\partial G$ such that
\begin{equation}\label{Kbound}
K\,\dist(x,\pa G) + \min_{\pa G} u \leq u(x),
\end{equation}
for every $x \in \ovr{G}$.

Moreover, if $x \in \partial X$, then
\begin{equation}\label{bound on dist}
\dist(x,\pa G) \leq C_* \usn,
\end{equation}
with
\begin{equation}\label{C4}
C_*=\frac{C+\diam(G)}{K},
\end{equation}
and where $C$ is given by \eqref{C3}.
\end{lm}

\begin{proof}
Let $v$ be such that
\begin{equation*}
\Delta v = -1 \ \mbox{ in } \ G,\quad v=\min_{\pa G} u \ \mbox{ on } \ \pa G;
\end{equation*}
by the comparison principle, $v \leq u$ on $\ovr{G}$. Since $v - \min_{\pa G} u = 0$ on $\pa G$,
applying \eqref{boundsdist} to $\pa G$ instead of $\pa \Om$ yields that
\begin{equation*} \label{K con v}
K\,\dist (x,\pa G) \leq v(x) - \min_{\pa G} u \ \mbox{ for any } \ x\in G;
\end{equation*}
 here, we have set $K=K_1$. Since $v \leq u$ we obtain \eqref{Kbound}.
\par
Now, we prove \eqref{bound on dist}. If $x \cdot \omega \geq m $ then $x \in \pa G$ and \eqref{bound on dist} clearly holds. If $x \cdot \omega < m $, we notice that $x^{m} \in \pa G$ and then, by Theorem \ref{th:abr},
\begin{equation*} \label{eq 5}
u(x) = w^m(x^{m}) + u(x^{m}) \leq C\,\usn + u(x^{m}).
\end{equation*}
Since $x^{m}\in \pa G$, then
\begin{equation*}
u(x^m)- \min_{\pa G} u \leq  \diam(G) \, \usn;
\end{equation*}
the last two inequalities then give that
\begin{equation*}\label{bound on dist proof2}
u(x) - \min_{\pa G} u \leq (C+\diam(G))\usn.
\end{equation*}
This inequality and \eqref{Kbound} give \eqref{bound on dist} at once.
\end{proof}

\par
For $s>0$, let $G_\shortparallel^s$ be the subset of $G$  {\it parallel} to $\pa G$
at distance $s$, i.e.
\begin{equation*}
G_\shortparallel^s=\{ x \in G :\ \dist(x,\pa G)>s\}.
\end{equation*}
We notice that $G_\shortparallel^s$ is connected for $s<\rho /2$; indeed,
any path in $G$ connecting any two points $x,y\in G_\shortparallel^s$ can be moved inwards into $G_\shortparallel^s$ by the normal field on $\pa G$.
\par
Next result is crucial to prove Theorem \ref{th:torsionstability}.

\begin{theorem}
\label{th:prop5}
Let
\begin{equation*}
\usn < \frac{\rho}{4C_*}.
\end{equation*}
Then, for every $s$ in the interval $(C_*\usn \, ,\rho/2)$, we have that
\begin{equation}
\label{XtOms}
G_\shortparallel^{s} \subset X \subset G.
\end{equation}
\end{theorem}

\begin{proof}
In order to prove \eqref{XtOms}, we proceed by contradiction. Since the maximal cap $G_{m}$ contains a ball of radius $\rho/2$, then $X$ intersects $G_\shortparallel^{s}$.
Let us assume that there exist a point $y \in G_\shortparallel^{s} \setminus X$ and take $x \in X \cap G_\shortparallel^{s}$. Since $G_\shortparallel^{s}$ is connected, we can join $x$ to $y$ with a path contained in $G_\shortparallel^{s}$. Let $z$ be the first point on this path that falls outside $X$; we notice that $z \in \pa X \cap G_\shortparallel^{s}$.

If $z\cdot \omega \geq m$, then $z \in \pa G$, that contradicts the fact that  $\dist(z,\pa G) > s$. If instead $z\cdot \omega < m$, a contradiction is reached by observing that, by Lemma \ref{lemma bound on dist}, $s<\dist(z,\pa G) \leq C_* \usn$.
\end{proof}

We have now all the ingredients to complete the proof of Theorem \ref{th:torsionstability}.

\vskip.2cm

\begin{proofA}
An admissible choice of the parameter $s$ in Theorem \ref{th:prop5} is
\begin{equation}\label{s}
s=2C_* \, \usn.
\end{equation}
Theorem \ref{th:prop5} then implies that, for any $x\in\pa G$, there is a point
$y\in\pa G$ such that
\begin{equation}
\label{cor1}
|x^{m}-y|\le 2s,
\end{equation}
(see \cite[Corollary 1]{ABR}).
It is important to observe that $s$ does not depend on the direction $\om.$
Thus, we can choose $\om$ to be each one of the coordinate directions
$e_1,\dots, e_N;$  by \eqref{definitions} and \eqref{m def}, these choices define $N$ hyperplanes $\pi_{m_1},\dots, \pi_{m_N}$
each in critical position with respect to each direction $e_1,\dots, e_N,$ respectively.
\par
Let $O$ be the point of intersection of the $N$ hyperplanes $\pi_{m_1},\dots, \pi_{m_N}$
and denote by $x^O$ the reflection $2 O-x$ of $x$ in $O;$ since $x^O$ is obtained by $N$
successive reflections with respect to the planes $\pi_{m_1},\dots, \pi_{m_N},$
by \eqref{cor1} one can infer that for any point $x\in\pa G$ there is a point
$y\in\pa G$ such that
\begin{equation*}
\label{cor2}
|x^O-y|\le 2Ns
\end{equation*}
(see \cite[Corollary 2]{ABR}).
\par
The point $O$ can now be chosen as the center of the balls $B_{r_i}$ and $B_{r_e}$ in
\eqref{stabil1}. In fact, since \eqref{cor1} holds, \cite[Proposition 6]{ABR} guarantees
that, for any direction $\om,$
\begin{equation*}
\label{cor3}
\dist(O,\pi_{m})\le 4N\, (1+\diam G)\, s
\end{equation*}
with $s$ given by \eqref{s}.
It is clear that \eqref{stabil1} holds with
$$
r_i=\min_{x\in\pa G}|x-O| \ \mbox{ and } \ r_e=\max_{x\in\pa G}|x-O|.
$$
Finally,
\cite[Proposition 7]{ABR} states that
$$
r_e-r_i\le 8N \, (1+\diam G)\, s.
$$
This last estimate and \eqref{s} give \eqref{stabil2}.
\end{proofA}

\begin{remark} \label{rem C}
We notice that the constant $C$ in Theorem \ref{th:torsionstability} is given by
\begin{equation*}
    C=16 N(1+\diam G)\,C_*.
\end{equation*}
Hence, the dependence of $C$ on $R$ is of the following type:
\begin{equation*}
C=O(A^{R^{-N}})\ \textmd{ as } R \to 0^+,
\end{equation*}
where $A>1$ is a constant and
\begin{equation*}
C=O(R) \ \textmd{ as } R \to +\infty.
\end{equation*}
\end{remark}

\setcounter{equation}{0}

\section{Semilinear equations} \label{section semilinear}

In this section, we shall show how Theorem \ref{th:torsionstability} can be extended to the case in
which \eqref{torsion} is replaced by \eqref{semilinear}. The structure of the proofs is the same: for this reason, we will limit ourselves to identify only the relevant passages that need to be modified.
In what follows, we will use the standard norms:
\begin{eqnarray*}
&&\nr u\nr_{C^1(\Om)}=\max_{\ovr\Om} |u|+\max_{\ovr\Om} |D u|, \\
&&\nr u\nr_{C^2(\Om)}=\nr u\nr_{C^1(\Om)}+\max_{\ovr\Om} |D^2 u|.
\end{eqnarray*}

Let $f$ be a locally Lipschitz continuous function with $f(0) \geq 0$ and consider the following problem
\begin{equation}\label{pb semilinear}
\begin{cases}
\Delta u + f(u) = 0 , & \textmd{in } \Om, \\
u>0, & \textmd{in } \Om, \\
u=0, & \textmd{on } \pa \Om,
\end{cases}
\end{equation}
where $\Om = G+ B_R$.

Also in this case,  if we add the overdetermination \eqref{parallel},
then the $G$ must be a ball and $u$ is radially symmetric. The details of this result
can be reconstructed  from \cite{CMS}.
\par
In the following theorem, we assume the domain $G$ as in Theorem \ref{th:torsionstability} .

\begin{theorem}
\label{th:semilinear}
Let $u\in C^2(\ovr{\Om})$ be the solution of \eqref{pb semilinear} with $-\pa u / \pa \nu \geq d_0 > 0$ on $\pa \Om$.
\par
If $R < \frac12\,d_0\,\| u\|^{-1}_{C^2(\Om)}$, then there exist constants $\veps, C>0$ such that, if
$$
\usn \leq \veps,
$$
then there are two concentric balls $B_{r_i}$ and $B_{r_e}$ such that \eqref{stabil1} and \eqref{stabil2} hold, that is
\begin{eqnarray*}
&B_{r_i}\subset\Om\subset B_{r_e}\quad\mbox{ and } \\
&r_e-r_i\le C\ \usn.
\end{eqnarray*}
The constants $\veps$ and $C$ only depend on $N$, $R$, the $C^{2,\alpha}$-regularity
of $\pa G$ and on upper bounds for the diameter of $G$, $\max_{\ovr\Om} |u|$ and
the Lipschitz constant of $f$.
\end{theorem}

As for Theorem \ref{th:torsionstability}, the proof of Theorem \ref{th:semilinear} is based on a quantitative study of the method of moving planes. In this case, the function $w^m$ defined by \eqref{wm} satisfies the conditions
\begin{equation*} \label{eq w semilinear}
\De w^{m}+c(x)\,w^{m}=0 \ \mbox{ and } \ w^{m}\ge 0 \ \mbox{ in } \Om_{m},
\end{equation*}
where for $x\in\Om_{m}$
\begin{equation*}
\label{defc}
c(x)=\left\{
\begin{array}{lll}
\ds\frac{f(u(x^{m}))-f(u(x))}{u(x^{m})-u(x)} &\mbox{ if } u(x^{m})\not= u(x),\\
\ds 0 &\mbox{ if } u(x^{m}) = u(x).
\end{array}
\right.
\end{equation*}
Notice that $c(x)$ is bounded by the Lipschitz constant $L$ of $f$ in the interval
$[0,\max\limits_{\ovr{\Om}}u].$

The following Lemma summarizes and generalizes the contents of Lemmas \ref{proposition Harnack} and \ref{proposition Harnack 2} to the present case.

\begin{lm}\label{lem Harnack semil}
Let $D_1,\: D_2$ and $R$ be as in Lemma \ref{proposition Harnack}. Assume that a function $w\in C^0(D_2) \cap C^2(D_2^+)$ satisfies the conditions
\begin{equation*}\label{eq 1 lemma semil}
\begin{cases}
\Delta w + c(x) w = 0, \quad w \geq 0, & \textmd{ in } D_2^+, \\
w=0, & \textmd{ on } \pa D_2^+ \cap \{x_1=0\},
\end{cases}
\end{equation*}
with $|c| \leq L$ in $D_2^+$.

Let $z \in \overline{D_1^+}$. Then, for any $x \in \overline{D_1^+} \setminus \{x_1 = 0 \}$, it holds that
\begin{eqnarray}
&& w(x) \leq C \frac{w(z)}{z_1}, \quad  \textmd{if } z_1>0,  \label{estim z1 >0 semilinear} \\
&& w(x) \leq C \frac{\partial w}{\pa x_1} (z),\quad  \textmd{if } z_1=0,  \label{estim z1 =0 semilinear}
\end{eqnarray}
where $C$ is a constant depending only on $N$, $R$, $L$ and the diameters of $D_1$ and $D_2$.
\end{lm}

\begin{proof}
Let $E_R$ be defined by \eqref{D e Er}. First, we prove \eqref{estim z1 >0 semilinear} and \eqref{estim z1 =0 semilinear} for $x\in E_R$ (as done in Lemma \ref{proposition Harnack}) and then we extend such estimates to any $x \in \overline{D_1^+} \setminus \{x_1 = 0 \}$ (as in Lemma \ref{proposition Harnack 2}). We follow step by step the proofs of Lemmas \ref{proposition Harnack} and \ref{proposition Harnack 2}.

The main ingredient in the proof of Lemma \ref{proposition Harnack} was Harnack's inequality. In the present case, \cite[Theorem 8.20]{GT} gives
\begin{equation} \label{harnack semil}
\sup_{B_{r/4}} w \leq C \inf_{B_{r/4}} w ,
\end{equation}
for any $B_r \subset D_1^+$; here, the constant $C$ can be bounded by $C_N^{\sqrt{N} + \sqrt{L R}}$, where $C_N$ is a constant depending only on $N$. Inequality \eqref{harnack semil} replaces \eqref{Harnack Harmonic 2} in this setting.

If we assume that case (i) of the proof of Lemma \ref{proposition Harnack} holds, we readily obtain
\begin{equation*}
\sup_{E_R} w \leq (\diam D_1) C^{({2\diam D_2} /R) ^N} \frac{w(z)}{z_1}.
\end{equation*}
If case (ii) or case (iii) hold, the conclusion follows by applying \cite[Proposition 2]{ABR} and \cite[Lemma 2]{ABR}, respectively.

To prove the analogous of Lemma \ref{proposition Harnack 2}, we need the equivalent of \eqref{bound Harnack I} for the semilinear case. This can be found, for instance, in \cite[Theorem 1.3]{BCN}, which gives
\begin{equation*}
\sup_{B_{R/2}^+} w \leq M u(\hat{x});
\end{equation*}
here, the constant $M$ depends only on $R$, $N$ and $L$. The rest of the proof is completely analogous to the one of Lemma \ref{proposition Harnack 2}.
\end{proof}

Theorem \ref{th:abr} and its proof apply also to the semilinear case. It is clear that the constant $MC$ appearing in \eqref{smallness} will now depend on $N$, $R$, $\diam G$, the $C^2$ regularity of $\pa G$, and $L$.

Lemma \ref{lemma bound on dist} is the last ingredient that needs some remodeling.

\begin{lm}
\label{lemma bound on dist semilin}
Let $u \in C^2(\overline{\Om})$ be a solution of \eqref{pb semilinear} with $-\pa u / \pa \nu \geq d_0 > 0$ on $\pa \Om$, where $\nu$ is the normal exterior to $\pa \Om$.
\par
If $R < \frac12\,d_0\,\| u\|^{-1}_{C^2(\Om)}$, then there exists a positive constant $K$ depending on $\diam G$, $\max_{\ovr\Om}|u|$, the $C^{2,\alpha}$ regularity of $\partial G$ and the Lipschitz constant  $L$ such that
\begin{equation*}\label{Kbound semilin}
K\,\dist(x,\pa G) + \min_{\pa G} u \leq u(x),
\end{equation*}
for every $x \in \ovr{G}$.
\end{lm}

\begin{proof}
The proof follows the steps of that of \cite[Proposition 4]{ABR}. Let $x \in G$ and assume that $\dist(x,\pa G) < \rho$, where $\rho$ is the optimal radius of the interior touching ball to $\pa G$. For such $x$, there exist $y \in \pa G$ and $z \in \pa \Om$ such that
\begin{equation*}
y-x= \dist(x,\pa G)\,\nu(y),\quad z-x= (\dist(x,\pa G)+R)\,\nu(z).
\end{equation*}
We notice that $\nu(y)=\nu(z)$, $|y-z|=R$ and $- \nabla u(z) \cdot \nu(z) \geq d_0$. The mean value theorem implies that
\begin{equation} \label{taylor nabla u}
-\nabla u(y) \cdot \nu(y) \geq d_0 -  \|u\|_{C^2(\Om)} R \geq \frac{d_0}{2},
\end{equation}
for every $R \leq \frac 12 d_0/\|u\|_{C^2(\Om)}$. Again, by Taylor expansion we have
\begin{equation*}
u(x) \geq u(y) + \nabla u(y) \cdot (x-y) - \frac{1}{2} \|u\|_{C^2(\Om)} |x-y|^2,
\end{equation*}
and from \eqref{taylor nabla u} we obtain that
\begin{equation*}
u(x) \geq u(y) +\frac{1}{4}\,d_0\,|x-y|  , \quad \textmd{for every } |x-y| \leq \frac12\, d_0\,\|u\|^{-1}_{C^2(\Om)},
\end{equation*}
which implies that
\begin{equation} \label{taylor u}
u(x) \geq \min_{\pa G} u + \frac{1}{4}\,d_0\, \dist(x,\pa G),
\end{equation}
for every $x$ such that $\dist(x,\pa G) \leq \frac12\, d_0\,\|u\|^{-1}_{C^2(\Om)}$.

Now, \eqref{taylor u} replaces formula (28) in the proof of \cite[Proposition 4]{ABR} and each argument of that proof can be repeated in our case leading to the conclusion.
\end{proof}

\begin{proofD}
The proof follows the outline of that of Theorem \ref{th:torsionstability}.
\par
For any fixed direction $\om$, the moving plane method provides a maximal cap $G_{m}$ where, thanks to Lemma \ref{lem Harnack semil}, we have the bound
\begin{equation*}
w^{m}\le C\, \usn \ \mbox{ in } \ G_m,
\end{equation*}
which is the analog of \eqref{smallness}. Here, $w^{m}$ is defined by \eqref{wm} and $C$ is a constant depending on $N$, $R$, $\diam G$, the $C^2$ regularity of $\pa G$ and on $L$.

Then, we define a symmetric set $X$ as done in Section \ref{section stability estimates}. We use Lemma \ref{lemma bound on dist semilin} and prove that
\begin{equation*}
\dist(x,\pa G) \leq C_* \, \usn,
\end{equation*}
for any $x \in \pa X$, where $C$ is as in \eqref{C4}. This says that, if $ \usn $ is small, then $X$ is almost equal to $G$; in particular we have that
\begin{equation*}
G_\shortparallel^{s} \subset X \subset G,
\end{equation*}
for every $s$ in the interval $(C_*\, \usn \, ,\rho/2)$ and whenever
\begin{equation*}
\usn < \frac{\rho}{4C_*}.
\end{equation*}
We notice that such estimates do not depend on the direction $\om$. By choosing $\om$ to be each of the coordinate directions $e_1,\ldots,e_N$, we define the approximate center of symmetry and the conclusion follows by repeating the argument of the proof of Theorem \ref{th:torsionstability}.
\end{proofD}

\end{document}